\documentclass[12pt]{article}
\usepackage{amsmath}
\usepackage[cp1250]{inputenc}
\usepackage[T1]{fontenc}
\usepackage{amssymb, latexsym}
\usepackage{amsthm}
\usepackage{amscd}
\usepackage{makeidx}
\usepackage{amsfonts,mathrsfs}
\usepackage{topcapt}
\usepackage{hyperref}
\hypersetup{colorlinks=true, debug=true, linkcolor=blue,
hypertexnames=false}
\let\originalforall=\forall
\renewcommand{\forall}{\mathop{\vcenter{\hbox{\Large$\originalforall$}}}}

\let\originalexists=\exists
\renewcommand{\exists}{\mathop{\vcenter{\hbox{\Large$\originalexists$}}}}
\newtheorem{de}{Definition}
\newtheorem{fakt}[de]{Fact}
\newtheorem{tw}[de]{Theorem}
\newtheorem{prop}[de]{Proposition}
\newtheorem{lem}[de]{Lemma}
\newtheorem{cor}[de]{Corollary}

\date{}

\title{Spectrally reasonable measures}

\author{Przemysław Ohrysko\thanks{The research of this author has been supported by National Science Centre, Poland grant no. 2014/15/N/ST1/02124}\\
Institute of Mathematics, Polish Academy of Sciences\\
00-956 Warszawa, Poland\\
E-mail: p.ohrysko@gmail.com \and
Michał Wojciechowski\\
Institute of Mathematics, Polish Academy of Sciences\\
00-956 Warszawa, Poland\\
E-mail: miwoj-impan@o2.pl}
\begin{document}

\title{Spectrally reasonable measures}

\maketitle
\renewcommand{\thefootnote}{}



\renewcommand{\thefootnote}{\arabic{footnote}}
\setcounter{footnote}{0}
\begin{abstract}
In this paper we investigate the problems related to measures with a natural spectrum (equal to the closure of the set of the values of the Fourier-Stieltjes transform). Since it is known that the set of all such measures does not have a Banach algebra structure we consider the set of all suitable perturbations called spectrally reasonable measures. In particular, we exhibit a broad class of spectrally reasonable measures which contains absolutely continuous ones. On the other hand, we show that except trivial cases all discrete (purely atomic) measures do not posses this property.
\end{abstract}
\section{Introduction}
We consider the Banach algebra $M(\mathbb{T})$ of complex Borel
regular measures on the circle group with $M_{d}(\mathbb{T})$ as a
closed subalgebra consisting of all discrete (purely atomic
measures). Here $\mathbb{T}$ is the circle group identified with the quotient group $\mathbb{R}/2\pi\mathbb{Z}$. Let $\mathfrak{M}(M(\mathbb{T}))$ denote the maximal
ideal space of $M(\mathbb{T})$, $\sigma(\mu)$ the spectrum of
$\mu\in M(\mathbb{T})$, $r(\mu)$ its spectral radius and
$\widehat{\mu}$ the Gelfand transform of $\mu$ (for basic facts in
commutative harmonic analysis and Banach algebras see for example
$\cite{katz}$ and $\cite{z}$) . We will treat the Fourier-Stieltjes transform of a measure as a restriction of its Gelfand
transform to $\mathbb{Z}$ and so $\widehat{\mu}(n)$ will stand for
the $n$-th Fourier-Stieltjes coefficient ($n\in\mathbb{Z}$). It
is an elementary fact that for all $n\in\mathbb{Z}$ we have
$\widehat{\mu}(n)\in\sigma(\mu)$ and since the spectrum of an
element in a unital Banach algebra is closed we have an inclusion
$\overline{\widehat{\mu}(\mathbb{Z})}\subset\sigma(\mu)$. However,
as was first observed by Wiener and Pitt in $\cite{wp}$ this
inclusion may be strict and this interesting spectral behavior is
now known as the Wiener-Pitt phenomenon (for the first precise
proof of the existence of the Wiener-Pitt phenomenon see
$\cite{schreider}$, for the most general result consult
$\cite{wil}$ and $\cite{r}$, an alternative approach is presented
in $\cite{graham}$). Therefore, after M. Zafran (see
$\cite{Zafran}$), we introduce the following notion.
\begin{de}
We say that a measure $\mu\in M(\mathbb{T})$ has a natural
spectrum, if
\begin{equation*}
\sigma(\mu)=\overline{\widehat{\mu}(\mathbb{Z})}=\overline{\{\widehat{\mu}(n):n\in\mathbb{Z}\}}.
\end{equation*}
The set of all such measures will be denoted by
$\mathscr{N}$.
\end{de}
It is proved in $\cite{Zafran}$ that this set is not closed under addition. This may be also deduced from the following result of Hatori and Sato (for the general proof for any compact abelian group see $\cite{hs}$, the elementary proof for the circle group is presented in $\cite{o}$), since $M_{d}(\mathbb{T})\subset\mathscr{N}$ and otherwise (if we have assumed that the set $\mathscr{N}$ is closed under addition) every measure would have a natural spectrum. In the formulation of Theorem $\ref{hs}$ the abbrevation $M(G)$ stands for the measure algebra on $G$, $M_{d}(G)$ denotes its closed subalgebra consisting of discrete measures and $\mathscr{N}(G)$ is the set of all measures with a natural spectrum in $M(G)$, i.e. the set of all $\mu\in M(G)$ with $\sigma(\mu)=\overline{\{\widehat{\mu}(\gamma):\gamma\in\Gamma\}}$ where $\Gamma$ is a dual group of $G$.
\begin{tw}[Hatori,Sato]\label{hs}
Let $G$ be a compact abelian group. Then
\begin{equation*}
M(G)=\mathscr{N}(G)+\mathscr{N}(G)+M_{d}(G).
\end{equation*}
\end{tw}
In the following proposition we obtain few properties of $\mathscr{N}$ which will be used later (the algebra $M(\mathbb{T})$ is endowed with a standard involution defined for every $\mu\in M(\mathbb{T})$ as $\widetilde{\mu}(E):=\overline{\mu(E)}$, where $E\subset\mathbb{T}$ is a Borel set).
\begin{prop}\label{mon}
The set of all measures with natural spectra is involutive ($\mu\in\mathscr{N}$ $\Rightarrow$ $\widetilde{\mu}\in\mathscr{N}$), closed subset of $M(\mathbb{T})$ which is also closed under multiplication by complex numbers. Moreover, this set is closed under an action of the functional calculus i. e. for $\mu\in\mathscr{N}$ and a holomorphic function defined on some open neighborhood of $\sigma(\mu)$ we have $f(\mu)\in\mathscr{N}$.
\end{prop}
\begin{proof}
Let us take $\mu\in\mathscr{N}$. Then it is obvious that $\alpha\mu\in\mathscr{N}$ for all $\alpha\in\mathbb{C}$. Moreover, $\sigma(\widetilde{\mu})=\overline{\sigma(\mu)}$ and $\widehat{\widetilde{\mu}}(n)=\overline{\widehat{\mu}}(n)$ for $n\in\mathbb{Z}$ shows involutivness of $\mathscr{N}$.
Let $(\mu_{n})_{n=1}^{\infty}\subset\mathscr{N}$ satisfy $||\mu_{n}-\mu||\rightarrow 0$ as $n\rightarrow\infty$ for some $\mu\in M(\mathbb{T})$ and fix $\varepsilon>0$. Then for any $\varphi\in\mathfrak{M}(M(\mathbb{T}))$ and $k\in\mathbb{Z}$ we have
\begin{gather*}
|\widehat{\mu}(\varphi)-\widehat{\mu}(k)|\leq |\widehat{\mu}(\varphi)-\widehat{\mu_{n}}(\varphi)|+|\widehat{\mu_{n}}(\varphi)-\widehat{\mu_{n}}(k)|+|\widehat{\mu_{n}}(k)-\widehat{\mu}(k)|\leq
\\2||\mu-\mu_{n}||+|\widehat{\mu_{n}}(\varphi)-\widehat{\mu_{n}}(k)|.
\end{gather*}
Now, we fix $n_{0}\in\mathbb{N}$ such that $||\mu-\mu_{n_{0}}||<\frac{\varepsilon}{3}$. Since $\mu_{n_{0}}\in\mathscr{N}$ there exists $k_{n_{0}}\in\mathbb{N}$ satisfying $|\widehat{\mu_{n_{0}}}(\varphi)-\widehat{\mu_{n_{0}}}(k_{n_{0}})|<\frac{\varepsilon}{3}$ and we get
\begin{equation*}
|\widehat{\mu}(\varphi)-\widehat{\mu}(k_{n_{0}})|<\varepsilon.
\end{equation*}
Since the spectrum of an element in a commutative Banach algebra is an image of its Gelfand transform, the first part of the proof is finished.
\\
Let us take $\mu\in\mathscr{N}$ and $f$ - a holomorphic function defined on some open neighborhood of $\sigma(\mu)$. Then $f$ acts on $\mu$ and so there exists $\nu:=f(\mu)$ such that
\begin{equation}\label{rach}
\varphi(\nu)=f(\varphi(\mu)), \text{ for all $\varphi\in\mathfrak{M}(M(\mathbb{T}))$}.
\end{equation}
The assertion $f(\mu)\in\mathscr{N}$ follows from the spectral mapping theorem ($\sigma(\nu)=f(\sigma(\mu))$) and an elementary fact that under continuous mapping the dense set in the domain goes to the dense set in the image.
\end{proof}
We are going to show now how additive properties can be related to multiplicative ones. As an introduction we will prove the following fact.
\begin{fakt}\label{mm}
The set $\mathscr{N}$ is not closed under convolution.
\end{fakt}
\begin{proof}
Suppose on the contrary that $\mathscr{N}$ is closed under convolution and let us take $\mu,\nu\in\mathscr{N}$ such that $\mu+\nu\notin\mathscr{N}$. We may also assume that $\nu$ is invertible because otherwise we could consider $\mu+(\nu+c\delta_{0})\notin\mathscr{N}$ where $c>0$ is big enough to ensure the invertibility of $\nu+c\delta_{0}$. Since $\nu\in\mathscr{N}$ is invertible we obtain $\nu^{-1}\in\mathscr{N}$ (the element $\nu^{-1}$ is an action of the function $f(z)=z^{-1}$ on $\nu$ and the result follows from Proposition $\ref{mon}$). Now, if $\mu\ast\nu^{-1}\notin\mathscr{N}$ then we are done. Hence we are allowed to assume that $\mu\ast\nu^{-1}\in\mathscr{N}$. But in this case we have
\begin{equation*}
\mu+\nu=(\mu\ast\nu^{-1}+\delta_{0})\ast\nu\in\mathscr{N}
\end{equation*}
which contradicts the choice of $\mu$ and $\nu$.
\end{proof}
\textbf{Acknowledgements.} The authors would like to express their warm gratitude to the referee for the careful reading of the paper and many useful suggestions.
\section{Reasonability}
As we saw, the set $\mathscr{N}$ does not have the Banach algebra structure so it is convenient to introduce the set of 'suitable perturbations' which will be our main point of interest.
\begin{de}
We say that a measure $\mu\in M(\mathbb{T})$ is \textbf{spectrally reasonable}, if $\mu+\nu\in\mathscr{N}$ for all $\nu\in\mathscr{N}$. The set of all spectrally reasonable measures will be denoted by $\mathscr{S}$.
\end{de}
It is clear from the definition that the set $\mathscr{S}\subset\mathscr{N}$ is closed under addition and multiplication by complex numbers. Before we show that it has a Banach algebra structure we will prove an auxiliary lemma.
\begin{lem}\label{przy}
Spectrally reasonable measures have the following properties:
\begin{enumerate}
    \item If $\mu\in\mathscr{S}$ and $\nu\in\mathscr{N}$, then $\mu\ast\nu\in\mathscr{N}$.
    \item If $\mu\in\mathscr{S}$ is invertible, then $\mu^{-1}\in\mathscr{S}$.
\end{enumerate}
\end{lem}
\begin{proof}
Let us take $\mu\in\mathscr{S}$ and consider first only invertible $\nu\in\mathscr{N}$. Then, by Proposition $\ref{mon}$ we have $\nu^{-1}\in\mathscr{N}$ (inverse of the element is a result of the action of the function $f(z)=z^{-1}$). From the definition of $\mathscr{S}$ we get $\mu+\nu^{-1}\in\mathscr{N}$. Now,
\begin{equation*}
\lambda\in\sigma(\mu\ast\nu)\Leftrightarrow 0\in\sigma(\mu\ast\nu-\lambda\delta_{0})\Leftrightarrow 0\in\sigma(\mu-\lambda\nu^{-1}).
\end{equation*}
Since the set $\mathscr{N}$ is closed under multiplication by scalars $-\lambda\nu^{-1}\in\mathscr{N}$ which by the definition of $\mathscr{S}$ leads to $\mu-\lambda\nu^{-1}\in\mathscr{N}$. Hence, there exists a sequence of integers $(n_{k})$ such that
\begin{equation*}
\lim_{k\rightarrow\infty}\left(\widehat{\mu}(n_{k})-\frac{\lambda}{\widehat{\nu}(n_{k})}\right)=0.
\end{equation*}
This is obviously equivalent to
\begin{equation*}
\lim_{k\rightarrow\infty}\widehat{(\mu\ast\nu)}(n_{k})=\lambda
\end{equation*}
which shows that $\mu\ast\nu\in\mathscr{N}$. For general $\nu\in\mathscr{N}$ we take $\alpha\in\mathbb{R}_{+}$ such that $\nu+\alpha\delta_{0}$ is invertible (we may put any $\alpha>r(\nu)$). Then
\begin{equation*}
\mu\ast\nu=\mu\ast(\nu+\alpha\delta_{0}-\alpha\delta_{0})=\mu\ast(\nu+\alpha\delta)-\alpha\mu.
\end{equation*}
From the earlier part of the proof $\mu\ast(\nu+\alpha\delta)\in\mathscr{N}$ and finally $\mu\ast\nu\in\mathscr{N}$ which gives the first claim of the lemma.
\\
We move to the second statement. Let us take an invertible $\mu\in\mathscr{S}$ and $\nu\in\mathscr{N}$. Then, similarly to previous arguments we have
\begin{equation*}
\lambda\in\sigma(\mu^{-1}+\nu)\Leftrightarrow 0\in\sigma(\mu^{-1}+\nu-\lambda\delta_{0})\Leftrightarrow 0\in\sigma(\mu\ast\nu+\delta_{0}-\lambda\mu).
\end{equation*}
From the first part of the lemma we have $\mu\ast\nu+\delta_{0}\in\mathscr{N}$ which gives the desired conclusion in exactly the same way as before.
\end{proof}
We are ready now to show that $\mathscr{S}$ has a Banach algebra structure.
\begin{tw}\label{pod}
The set $\mathscr{S}$ is closed, unital $^{\ast}$-subalgebra of $M(\mathbb{T})$.
\end{tw}
\begin{proof}
Closedness and involutivness of $\mathscr{S}$ follows directly from Proposition $\ref{mon}$. Of course, $\delta_{0}\in\mathscr{S}$ and so it is enough to prove that if $\mu_{1},\mu_{2}\in\mathscr{S}$, then
$\mu_{1}\ast\mu_{2}\in\mathscr{S}$. Let us take $\nu\in\mathscr{N}$ and assume first that $\mu_{2}$ is invertible. Then
\begin{equation*}
\lambda\in\sigma(\mu_{1}\ast\mu_{2}+\nu)\Leftrightarrow 0\in\sigma(\mu_{1}\ast\mu_{2}+\nu-\lambda\delta_{0})
\Leftrightarrow 0\in\sigma(\mu_{1}-\lambda\mu_{2}^{-1}+\nu\ast\mu_{2}^{-1}).
\end{equation*}
From the previous lemma (second part) $\mu_{2}^{-1}\in\mathscr{S}$ and so $\mu_{1}-\lambda\mu_{2}^{-1}\in\mathscr{S}$. Moreover, from the first part of the last lemma $\nu\ast\mu_{2}^{-1}\in\mathscr{N}$ which shows $\mu_{1}-\lambda\mu_{2}^{-1}+\nu\ast\mu_{2}^{-1}\in\mathscr{N}$ and we are able to proceed analogously as in the proof of the lemma. For general $\mu_{2}$ we take once again $\alpha\in\mathbb{R}_{+}$ such that $\mu_{2}+\alpha_{\delta_{0}}$ is invertible and then
\begin{equation*}
\mu_{1}\ast\mu_{2}+\nu=\mu_{1}\ast(\mu_{2}+\alpha\delta_{0})-\alpha\mu_{1}+\nu.
\end{equation*}
From the first part we obtain $\mu_{1}\ast(\mu_{2}+\alpha\delta_{0})-\alpha\mu_{1}\in\mathscr{S}$ which finishes the proof.
\end{proof}
Now, we will examine other relevant features of $\mathscr{S}$ (the abbreviation $\mathfrak{M}(\mathscr{S})$ stands for the Gelfand space of $\mathscr{S}$).
\begin{prop}\label{sym}
The algebra $\mathscr{S}$ is a symmetric Banach $^{\ast}$-algebra, i.e.
\begin{equation*}
\varphi(\widetilde{\mu})=\overline{\varphi(\mu)},\text{ for all $\varphi\in\mathfrak{M}(\mathscr{S})$ and $\mu\in\mathscr{S}$}.
\end{equation*}
\end{prop}
\begin{proof}
This is not difficult since for $\mu=\widetilde{\mu}\in\mathscr{S}$ we know that $\sigma(\mu)\subset\mathbb{R}$ which gives
\begin{equation*}
\varphi(\widetilde{\mu})=\varphi(\mu)=\overline{\varphi(\mu)}\text{ for all $\varphi\in\mathfrak{M}(\mathscr{S})$}.
\end{equation*}
For general $\mu\in\mathscr{S}$ we use standard decomposition into hermitian and anti-hermitian part
\begin{equation*}
\mu=\frac{\mu+\widetilde{\mu}}{2}+i\frac{\mu-\widetilde{\mu}}{2i}
\end{equation*}
and the result follows from the previous argument.
\end{proof}
The last proposition leads to the corollary which sheds some light on the structure of $\mathfrak{M}(\mathscr{S})$.
\begin{tw}
The set $\mathbb{Z}$ identified with functionals $\mu\mapsto\widehat{\mu}(n)$ is dense in $\mathfrak{M}(\mathscr{S})$.
\end{tw}
\begin{proof}
Let $\widehat{\mathscr{S}}=\{\widehat{\mu}:\mu\in\mathscr{S}\}$. The assertion of Proposition $\ref{sym}$ implies that $\widehat{\mathscr{S}}\subset C(\mathfrak{M}(\mathscr{S}))$ is a self-adjoint subalgebra which contains a constant function. Hence from the Stone - Weierstrass theorem $\widehat{\mathscr{S}}$ is dense in $C(\mathfrak{M}(\mathscr{S}))$.
\\
Let us assume on the contrary that $\overline{\mathbb{Z}}\neq\mathfrak{M}(\mathscr{S})$. Then from the Urysohn's lemma (in fact it follows just from the complete regularity of $\mathfrak{M}(\mathscr{S})$) there exists $f\in C(\mathfrak{M}(\mathscr{S}))$ such that $f|_{\overline{\mathbb{Z}}}\equiv 0$ and $f(\varphi)=1$ for some $\varphi\in\mathfrak{M}(\mathscr{S})\setminus\overline{\mathbb{Z}}$. Using the density of $\widehat{\mathscr{S}}$ in $C(\mathfrak{M}(\mathscr{S}))$ we find $\mu\in\mathscr{S}$ such that $||\widehat{\mu}-f||_{C(\mathfrak{M}(\mathscr{S}))}<\frac{1}{3}$. It gives $|\widehat{\mu}|_{\overline{\mathbb{Z}}}|<\frac{1}{3}$ and $|\widehat{\mu}(\varphi)|>\frac{2}{3}$. This is impossible because $\mu\in\mathscr{N}$ and so there exists a sequence $(n_{k})$ satisfying
\begin{equation*}
\lim_{k\rightarrow\infty}\widehat{\mu}(n_{k})=\widehat{\mu}(\varphi).
\end{equation*}
\end{proof}
\section{All Zafrans are reasonable}
In this section we will determine some members of $\mathscr{S}$. Let us recall first the main theorem of Zafran's paper $\cite{Zafran}$ (here $M_{0}(\mathbb{T})$ is an ideal in $M(\mathbb{T})$ consisting of all measures with the Fourier-Stieltjes transforms vanishing at infinity).
\begin{tw}
Let $\mathscr{C}=\{\mu\in
M_{0}(\mathbb{T}):\sigma(\mu)=\widehat{\mu}(\mathbb{Z})\cup\{0\}\}$.
\begin{enumerate}
    \item If
    $\varphi\in\mathfrak{M}(M_{0}(\mathbb{T}))\setminus\mathbb{Z}$,
    then $\varphi(\mu)=0$ for all $\mu\in\mathscr{C}$.
    \item $\mathscr{C}$ is closed ideal in $M_{0}(\mathbb{T})$.
    \item $\mathfrak{M}(\mathscr{C})=\mathbb{Z}$.
\end{enumerate}
\end{tw}
We would like to show that all measures from the Zafran's ideal $\mathscr{C}$ are spectrally reasonable. This result is known (see Theorem 20 in $\cite{ow}$) but we will prove it again in greater generality. To do so, we need to remind the definition of joint spectrum and basic properties of multi-variable functional calculus.
\begin{de}
Let $A$ be a unital, commutative Banach algebra and let $\mathfrak{M}(A)$ be its maximal ideal space. Then, for $x,y\in A$ we define a joint spectrum $\sigma(x,y)$ as follows:
\begin{equation*}
\sigma(x,y)=\{(\varphi(x),\varphi(y)):\varphi\in\mathfrak{M}(A)\}\subset\mathbb{C}^{2}.
\end{equation*}
\end{de}
It is now time to cite the theorem on multi-variable functional calculus (for a proof consult $\cite{ek}$ or $\cite{z}$).
\begin{tw}
Let $A$ be a unital, commutative Banach algebra. Then, for $x,y\in A$ and a holomorphic function $f$ of two variables defined on an open set $U\subset\mathbb{C}^{2}$ containing $\sigma(x,y)$ there exists an element $f(x,y)\in A$ such that
\begin{equation*}
\varphi(f(x,y))=f(\varphi(x),\varphi(y)),\text{ for all $\varphi\in\mathfrak{M}(A)$}.
\end{equation*}
\end{tw}
We will introduce now an auxiliary definition which will be useful in the sequel.
\begin{de}
Let $\mu,\nu\in M(\mathbb{T})$.
\begin{itemize}
  \item We say that $\sigma(\mu,\nu)$ is natural, if
    \begin{equation*}
    \sigma(x,y)=\overline{\{(\widehat{\mu}(n),\widehat{\nu}(n)):n\in\mathbb{Z}\}}.
    \end{equation*}
  \item We say that $\mu$ is \textbf{highly reasonable}, if $\sigma(\mu,\nu)$ is natural for every $\nu\in\mathscr{N}$.
\end{itemize}
\end{de}
Let us prove a very simple fact relating our new definitions with old ones.
\begin{fakt}\label{hr}
If $\mu\in M(\mathbb{T})$ is highly reasonable then $\mu$ is spectrally reasonable.
\end{fakt}
\begin{proof}
Let $\nu\in\mathscr{N}$. Then by the assumption $\sigma(\mu,\nu)$ is natural. We define $f:\mathbb{C}^{2}\mapsto\mathbb{C}$ by the formula $f(x,y)=x+y$ - it is clearly an entire function and our lemma will be proved if we show $f(\mu,\nu)\in\mathscr{N}$. Using the definition of functional calculus we have
\begin{equation*}
\varphi(f(\mu,\nu))=f(\varphi(\mu),\varphi(\nu))\text{ for all $\varphi\in\mathfrak{M}(M(\mathbb{T}))$}.
\end{equation*}
Now, since $(\varphi(\mu),\varphi(\nu))\in\sigma(\mu,\nu)$ there exists a sequence of integers $(n_{k})_{k=1}^{\infty}$ such that
\begin{equation*}
\lim_{k\rightarrow\infty}(\widehat{\mu}(n_{k}),\widehat{\nu}(n_{k}))=(\varphi(\mu),\varphi(\nu)).
\end{equation*}
Once again, by the equation defining the action of the functional calculus we obtain
\begin{gather*}
\varphi(f(\mu,\nu))=f(\varphi(\mu),\varphi(\nu))=\\
=\lim_{k\rightarrow\infty}f(\widehat{\mu}(n_{k}),\widehat{\nu}(n_{k}))=\lim_{k\rightarrow\infty}\widehat{f(\mu,\nu)}(n_{k}),
\end{gather*}
for all $\varphi\in\mathfrak{M}(M(\mathbb{T}))$ which finishes the proof.
\end{proof}
We also need a lemma which clarifies the nature of isolated points in the joint spectrum.
\begin{lem}\label{izz}
Let $\mu,\nu\in M(\mathbb{T})$. If $(\lambda_{1},\lambda_{2})\in\sigma(\mu,\nu)$ is an isolated point of $\sigma(\mu,\nu)$ then there exists $n\in\mathbb{Z}$ such that $(\lambda_{1},\lambda_{2})=(\widehat{\mu}(n),\widehat{\nu}(n))$.
\end{lem}
\begin{proof}
Since $(\lambda_{1},\lambda_{2})$ is an isolated point there exists two open sets $U_{1},U_{2}\subset\mathbb{C}^{2}$ such that $U_{1}\cap\sigma(\mu,\nu)=\{(\lambda_{1},\lambda_{2})\}$, $\sigma(\mu,\nu)\subset U_{1}\cup U_{2}$ and $U_{1}\cap U_{2}=\emptyset$. Suppose now on the contrary that $(\lambda_{1},\lambda_{2})\neq (\widehat{\mu}(n),\widehat{\nu}(n))$ for every $n\in\mathbb{Z}$. Then we define a function $f:U_{1}\cup U_{2}\mapsto\mathbb{C}$ by the formula
\begin{equation*}
f(z)=\left\{\begin{array}{c}
         1\text{ for }z\in U_{1}, \\
         0\text{ for }z\in U_{2}.
       \end{array}\right.
\end{equation*}
Clearly, $f$ is a holomorphic function on $U_{1}\cup U_{2}$ and applying the theorem on multi-variable functional calculus we obtain a measure $\tau:=f(\mu,\nu)$ such that
\begin{equation*}
\varphi(\tau)=f(\varphi(\mu),\varphi(\nu)),\text{ for all $\varphi\in\mathfrak{M}(M(\mathbb{T}))$}.
\end{equation*}
From the definition of $f$ we have $\widehat{\tau}(n)=f(\widehat{\mu}(n),\widehat{\nu}(n))=0$ for every $n\in\mathbb{Z}$ which implies $\tau=0$. On the other hand, there exists $\varphi\in\mathfrak{M}(M(\mathbb{T}))$ satisfying $(\varphi(\mu),\varphi(\nu))=(\lambda_{1},\lambda_{2})$. This shows
\begin{equation*}
\varphi(\tau)=f(\varphi(\mu),\varphi(\nu))=f(\lambda_{1},\lambda_{2})=1,
\end{equation*}
which is a contradiction.
\end{proof}
We formulate a corrolary which is contained in the Zafran's paper but which also follows immediately from our lemma by putting $\nu=0$.
\begin{cor}\label{iz}
Let $\mu\in M(\mathbb{T})$. If $\lambda\in\sigma(\mu)$ is an isolated point of $\sigma(\mu)$, then $\lambda\in\widehat{\mu}(\mathbb{Z})$.
\end{cor}
We are ready now to prove the main theorem of this section.
\begin{tw}
$\mathscr{C}\subset\mathscr{S}$.
\end{tw}
\begin{proof}
Let $\mu\in\mathscr{C}$ and $\nu\in\mathscr{N}$. By Fact $\ref{hr}$ it is enough to show that $\sigma(\mu,\nu)$ is natural. Since multiplicative-linear functionals extends from ideals we have the following decomposition of $\mathfrak{M}(M(\mathbb{T}))$ (check $\cite{ek}$):
\begin{gather*}
\mathfrak{M}(M(\mathbb{T}))=\mathfrak{M}(M_{0}(\mathbb{T}))\cup h(M_{0}(\mathbb{T}))\\
\text{ where }h(M_{0}(\mathbb{T}))=\{\varphi\in\mathfrak{M}(M(\mathbb{T})):\varphi|_{M_{0}(\mathbb{T})}\equiv 0\}.
\end{gather*}
Using the above fact and the theorem of Zafran we get
\begin{gather*}
\sigma(\mu,\nu)=\{(\varphi(\mu),\varphi(\nu)):\varphi\in\mathfrak{M}(M(\mathbb{T}))\}=\\
=\{(\varphi(\mu),\varphi(\nu)):\varphi\in\mathfrak{M}(M_{0}(\mathbb{T}))\}\cup\{(\varphi(\mu),\varphi(\nu)):\varphi\in h(M_{0}(\mathbb{T}))\}=\\
=\{(\varphi(\mu),\varphi(\nu)):\varphi\in\mathfrak{M}(M_{0}(\mathbb{T}))\setminus\mathbb{Z}\}\cup\{(\widehat{\mu}(n),\widehat{\nu}(n)):n\in\mathbb{Z}\}\cup\\
\cup\{(\varphi(\mu),\varphi(\nu)):\varphi\in h(M_{0}(\mathbb{T}))\}=\\
=\{(\widehat{\mu}(n),\widehat{\nu}(n)):n\in\mathbb{Z}\}\cup\{(0,\varphi(\nu)):\varphi\in\mathfrak{M}(M(\mathbb{T}))\setminus\mathbb{Z}\}.
\end{gather*}
Let us take a point $(0,\varphi(\nu))\in\sigma(\mu,\nu)$ for $\varphi\in\mathfrak{M}(M(\mathbb{T}))\setminus\mathbb{Z}$. If $(0,\varphi(\nu))$ is an isolated point of $\sigma(\mu,\nu)$ then by Lemma $\ref{izz}$ we are done because $(0,\varphi(\nu))\in\{(\widehat{\mu}(n),\widehat{\nu}(n)):n\in\mathbb{Z}\}$.
In the second case (the point $(0,\varphi(\nu))$ is an accumulation point of $\sigma(\mu,\nu)$) there exists a sequence $(\psi_{n})\subset\mathfrak{M}(M(\mathbb{T}))$ such that
\begin{equation*}
\lim_{n\rightarrow\infty}(\psi_{n}(\mu),\psi_{n}(\nu))=(0,\varphi(\nu)).
\end{equation*}
By the assumption at least one of the sequences $(\psi_{n}(\mu))_{n=1}^{\infty}$, $(\psi_{n}(\nu))_{n=1}^{\infty}$ consists of infinitely many different numbers. If the sequence $(\psi_{n}(\mu))_{n=1}^{\infty}$ consists of infinitely many different numbers then by the theorem of Zafran we can assume, without losing the generality, that $(\psi_{n})\subset\mathbb{Z}\subset\mathfrak{M}(M(\mathbb{T}))$ and the proof is finished. In the second case the sequence $(\psi_{n}(\nu))_{n=1}^{\infty}$ has infinitely many values and we see that $\varphi(\nu)$ is not an isolated point of $\sigma(\nu)$. Then, since $\nu\in\mathscr{N}$ we can find a sequence of distinct integers $(n_{k})_{k=1}^{\infty}$ such that $\widehat{\nu}(n_{k})\rightarrow\varphi(\nu)$ as $k\rightarrow\infty$. But $\mu$ belongs to $\mathscr{C}$ and in particular $\mu\in M_{0}(\mathbb{T})$ which gives
\begin{equation*}
\lim_{k\rightarrow\infty}(\widehat{\mu}(n_{k}),\widehat{\nu}(n_{k}))=(0,\varphi(\nu))
\end{equation*}
and finishes the proof.
\end{proof}
\section{Discrete measures are not reasonable}
In this section we will prove that there are no spectrally reasonable discrete measures except constant multiples of $\delta_{0}$ which seems to be quite surprising. We start with the following lemma which is a partial converse of the fact that the functional calculus maps measures with natural spectra to measures with natural spectra.

\begin{lem}\label{row}
Let $\mu\in M(\mathbb{T})$ and let $f$ be a holomorphic function defined on some open neighborhood of $\sigma(\mu)$. Assume that $f:\sigma(\mu)\mapsto f(\sigma(\mu))$ bijective. Then, if $f(\mu)\in\mathscr{N}$, we have $\mu\in\mathscr{N}$.
\end{lem}
\begin{proof}
Let us take $\lambda\in\sigma(\mu)$. Then by the assumption for $f(\lambda)\in f(\sigma(\mu))=\sigma(f(\mu))$ there exists a sequence $(n_{k})_{k\in\mathbb{N}}$ such that
\begin{equation*}
\lim_{k\rightarrow\infty}\widehat{f(\mu)}(n_{k})=\lim_{k\rightarrow\infty}f(\widehat{\mu}(n_{k}))=f(\lambda).
\end{equation*}
The sequence $(\widehat{\mu}(n_{k}))_{k\in\mathbb{N}}$ is a bounded sequence of complex numbers and so we can choose a convergent subsequence
\begin{equation*}
\lim_{l\rightarrow\infty}\widehat{\mu}(n_{k_{l}})=\alpha\in\mathbb{C}.
\end{equation*}
Since holomorphic functions are continuous we have
\begin{equation*}
\lim_{l\rightarrow\infty}\widehat{f(\mu)}(n_{k_{l}})=\lim_{l\rightarrow\infty}f(\widehat{\mu}(n_{k_{l}}))=f(\alpha).
\end{equation*}
Hence $f(\alpha)=f(\lambda)$ and by injectivity of $f$ we have $\lambda=\alpha\in\overline{\widehat{\mu}(\mathbb{Z})}$ which finishes the proof.
\end{proof}
The next proposition is inspired by the results of Hatori and Sato from \cite{hs}.
\begin{prop}\label{dn}
Let $\mu\in M(\mathbb{T})$ and let $r:=r(\mu)$. Then, for every measure $\nu\in M(\mathbb{T})$ such that
\begin{equation}\label{du}
\overline{\widehat{\nu}(\mathbb{Z})}=r\overline{\mathbb{D}}
\end{equation}
and
\begin{equation}\label{zer}
\mu\ast\nu=0,
\end{equation}
we have $\mu+\nu\in\mathscr{N}$.
\end{prop}
\begin{proof}
There exists, of course, $n\in\mathbb{Z}$ such that $\widehat{\mu}(n)=0$ (otherwise from $(\ref{zer})$, $\nu=0$ which is excluded by $(\ref{du})$), also from $(\ref{zer})$, for every $\varphi\in\mathfrak{M}(M(\mathbb{T}))$ we have
\begin{equation}\label{jd}
\varphi(\mu)=0\text{ or }\varphi(\nu)=0
\end{equation}
which gives
\begin{gather*}
\sigma(\mu)=\{\varphi(\mu):\varphi\in\mathfrak{M}(\mathbb{T})\}=\{\varphi(\mu):\varphi(\nu)=0\}\cup\{0\},
\\
\sigma(\nu)=\{\varphi(\nu):\varphi(\mu)=0\}\cup\{0\}
\end{gather*}
Applying consecutively $(\ref{jd})$, the last assertion and $(\ref{du})$ we obtain
\begin{gather*}
\sigma(\mu+\nu)\cup\{0\}=\{\varphi(\mu+\nu):\varphi\in\mathfrak{M}(M(\mathbb{T}))\}\cup\{0\}=\\
\{\varphi(\mu):\varphi(\nu)=0\}\cup\{\varphi(\nu):\varphi(\mu)=0\}\cup\{0\}=\sigma(\mu)\cup\sigma(\nu)=r\overline{\mathbb{D}}.
\end{gather*}
The spectrum of an element in a unital Banach algebra is closed so $\sigma(\mu+\nu)=r\overline{\mathbb{D}}$. Analogously,
\begin{equation*}
\left(\widehat{\mu}+\widehat{\nu}\right)(\mathbb{Z})\cup\{0\}=\widehat{\mu}(\mathbb{Z})\cup\widehat{\nu}(\mathbb{Z})\cup\{0\}.
\end{equation*}
Hence,
\begin{equation*}
r\overline{\mathbb{D}}=\overline{\widehat{\nu}(\mathbb{Z})}\subset \overline{\left(\widehat{\mu}+\widehat{\nu}\right)(\mathbb{Z})\cup\{0\}}.
\end{equation*}
Now, $0$ is not an isolated point of $\overline{\left(\widehat{\mu}+\widehat{\nu}\right)(\mathbb{Z})\cup\{0\}}$ and we get
\begin{equation*}
r\overline{\mathbb{D}}\subset\overline{\left(\widehat{\mu}+\widehat{\nu}\right)(\mathbb{Z})\cup\{0\}}=\\
\overline{\left(\widehat{\mu}+\widehat{\nu}\right)(\mathbb{Z})}\subset\sigma(\mu+\nu)=r\overline{\mathbb{D}}.
\end{equation*}
This precisely means
\begin{equation*}
\sigma(\mu+\nu)=\overline{\left(\widehat{\mu}+\widehat{\nu}\right)(\mathbb{Z})},
\end{equation*}
which finishes the proof.
\end{proof}
The following lemma is also suitable for our investigations.
\begin{lem}\label{zn}
Let $\mu,\nu\in\mathscr{N}$ be such that $\mu\ast\nu=0$. Then
$\mu+\nu\in \mathscr{N}$.
\end{lem}
\begin{proof}
Once again, for every $\varphi\in\mathfrak{M}(M(\mathbb{T}))$, we have
\begin{equation*}
\varphi(\mu)=0\text{ or }\varphi(\nu)=0,
\end{equation*}
which gives
\begin{equation*}
\sigma(\mu+\nu)\cup\{0\}=\sigma(\mu)\cup\sigma(\nu)\cup\{0\}.
\end{equation*}
In the same manner,
\begin{equation*}
\left(\widehat{\mu}+\widehat{\nu}\right)(\mathbb{Z})\cup\{0\}=\widehat{\mu}(\mathbb{Z})\cup\widehat{\nu}(\mathbb{Z})\cup\{0\}.
\end{equation*}
Now, we obtain
\begin{equation*}
\overline{\left(\widehat{\mu}+\widehat{\nu}\right)(\mathbb{Z})}\cup\{0\}=\overline{\widehat{\mu}(\mathbb{Z})}\cup\overline{\widehat{\nu}}\cup\{0\}
=\sigma(\mu)\cup\sigma(\nu)\cup\{0\}=\sigma(\mu+\nu)\cup\{0\}.
\end{equation*}
If $0\in\overline{\left(\widehat{\mu}+\widehat{\nu}\right)(\mathbb{Z})}$ then we are done because
\begin{equation*}
\sigma(\mu+\nu)\cup\{0\}=\overline{\left(\widehat{\mu}+\widehat{\nu}\right)(\mathbb{Z})}\cup\{0\}
=\overline{\left(\widehat{\mu}+\widehat{\nu}\right)(\mathbb{Z})}\subset\sigma(\mu+\nu),
\end{equation*}
which shows that $0\in\sigma(\mu+\nu)$ and provides $\sigma(\mu+\nu)=\overline{\left(\widehat{\mu}+\widehat{\nu}\right)(\mathbb{Z})}$.
\\
If $0\notin \overline{\left(\widehat{\mu}+\widehat{\nu}\right)(\mathbb{Z})}$, then $0\notin\sigma(\mu+\nu)$ as otherwise $\{0\}=\sigma(\mu+\nu)\setminus\overline{\left(\widehat{\mu}+\widehat{\nu}\right)(\mathbb{Z})}$ which is impossible since the latter set does not contain isolated points provided it is non-empty (see Corrolary $\ref{iz}$).
\end{proof}
We are able to prove now that $\delta_{\alpha}\notin\mathscr{S}$ for $\alpha\in\left(\mathbb{Q}/2\pi\mathbb{Z}\right)\setminus\{0\}$ (the subgroup $\left(\mathbb{Q}/2\pi\mathbb{Z}\right)$ consists of all torsion elements in $\mathbb{T}$).
\begin{tw}
Let $\alpha\in\left(\mathbb{Q}/2\pi\mathbb{Z}\right)\setminus\{0\}$. Then
$\delta_{\alpha}\notin\mathscr{S}$.
\end{tw}
\begin{proof}
Let us start from the case $\alpha=\pi$. Assume that
$\delta_{\pi}\in\mathscr{S}$. Then, of course,
$\frac{\delta_{0}+\delta_{\pi}}{2}\in\mathscr{S}$. Consider the
Riesz product
\begin{equation*}
R=\prod_{k=1}^{\infty}\left(1+\cos (4^{k}t)\right).
\end{equation*}
For such $R$ we have
\begin{equation*}
\mathrm{supp}\widehat{R}\subset 4\mathbb{Z}\subset 2\mathbb{Z}\text{ but }R\notin\mathscr{N}\text{ (see $\cite{graham}$)}.
\end{equation*}
Let us define a measure $\mu$ by the formula
\begin{equation*}
\mu=R+\frac{\delta_{0}-\delta_{\pi}}{2}\ast\frac{\delta_{\alpha}+\delta_{\beta}}{2}
\end{equation*}
where $\alpha,\beta\in\mathbb{T}$ are such that the set
$\{1,\alpha,\beta\}$ is linearly independent over $\mathbb{Q}$.
Then, for
$\nu=\frac{\delta_{0}-\delta_{\pi}}{2}\ast\frac{\delta_{\alpha}+\delta_{\beta}}{2}$
we have (this follows easily from the classical Kronecker's theorem, for details consult $\cite{o}$)
\begin{equation*}
\overline{\widehat{\nu}(\mathbb{Z})}=\overline{\mathbb{D}}.
\end{equation*}
Moreover, since $\frac{\delta_{0}-\delta_{\pi}}{2}\ast R=0$ we
obtain from Theorem $\ref{dn}$ that $\mu\in\mathscr{N}$. Now from
the first part of Lemma $\ref{przy}$ we get
$\frac{\delta_{0}+\delta_{\pi}}{2}\ast \mu\in\mathscr{N}$. But
\begin{equation*}
\mu\ast\frac{\delta_{0}+\delta_{\pi}}{2}=R\ast\frac{\delta_{0}+\delta_{\pi}}{2}=R
\end{equation*}
which is a contradiction.
\\
Since for $n\in\mathbb{N}$ we
have $\delta_{\frac{k\pi}{n}}^{\ast n}=\delta_{\pi}$ for any odd
$k$ and the set $\mathscr{S}$ of all reasonable measures is an algebra, we obtain
$\delta_{\frac{k\pi}{n}}\notin\mathscr{S}$ for any odd $k$.
\\
For any $\alpha=\frac{k}{l}\pi$, where $k,l$ are coprime natural
numbers such that $\frac{k}{l}\in (0,2)$ and $k$ is even we
proceed similarly. Let $\nu$ denote the normalized Haar measure of
the finite subgroup generated by $\alpha$ (this subgroup is of order $l$) in $\mathbb{T}$ . Then it is elementary to check that
$\widehat{\nu}(n)=1$ for $n\in l\mathbb{Z}$ and zero otherwise. As
before, we have $\frac{\delta_{0}-\delta_{\alpha}}{2}\ast \nu =0$.
Assume now that $\delta_{\alpha}\in\mathscr{S}$. Then since $\nu$
belongs to the algebra generated by $\delta_{\alpha}$ we have
$\nu\in\mathscr{S}$. We consider the measure $\mu\in\mathbb{T}$
given by the formula
\begin{equation*}
\mu=\prod_{n=1}^{\infty}(1+\cos
(l^{n}t))+\frac{\delta_{0}-\delta_{\alpha}}{2}\ast\frac{\delta_{\beta}+\delta_{\gamma}}{2},
\end{equation*}
where $\beta,\gamma\in\mathbb{T}$ are chosen in such a way that
the set $\{1,\beta,\gamma\}$ is linearly independent over
$\mathbb{Q}$. Basing once again on Theorem $\ref{dn}$ and Lemma
$\ref{przy}$ we obtain a contradiction.
\end{proof}
It is not difficult to extend the last result to the case of
finite sums of Dirac deltas at rational points but we will not provide any details since this case will be covered in Propositon $\ref{ds}$.
\\
In order to go further we need the following theorem of F.Parreu from $\cite{p}$.
\begin{tw}[Parreau]\label{par}
There exists $\tau\in M(\mathbb{T})$ such that $\sigma(\tau)\subset\mathbb{R}$ but $\tau\notin\mathscr{N}$.
\end{tw}
Let $T^{\ast}=\{z\in\mathbb{C}:|z|=1\}$. We state the following theorem whose proof is very similar to Proposition $\ref{dn}$.
\begin{tw}\label{dziu}
Let $\mu\in M(\mathbb{T})$ satisfies $\sigma(\mu)\subset T^{\ast}\cup\{0\}$. Then, for every measure $\nu\in M(\mathbb{T})$ such that
\begin{equation}
\overline{\widehat{\nu}(\mathbb{Z})}=T^{\ast}\cup\{0\}
\end{equation}
and
\begin{equation}
\mu\ast\nu=0,
\end{equation}
we have $\mu+\nu\in\mathscr{N}$.
\end{tw}
\begin{proof}
Once again,
\begin{equation*}
\sigma(\mu+\nu)\cup\{0\}=\sigma(\mu)\cup\sigma(\nu)\cup\{0\}=T^{\ast}\cup\{0\}.
\end{equation*}
Also,
\begin{equation*}
\left(\widehat{\mu}+\widehat{\nu}\right)(\mathbb{Z})\cup\{0\}=\widehat{\mu}(\mathbb{Z})\cup\widehat{\nu}(\mathbb{Z})\cup\{0\}.
\end{equation*}
This leads to
\begin{equation*}
T^{\ast}\cup\{0\}=\overline{\widehat{\nu}(\mathbb{Z})}\subset\overline{\left(\widehat{\mu}+\widehat{\nu}\right)(\mathbb{Z})}\cup\{0\}\\
\subset\sigma(\mu+\nu)\cup\{0\}=T^{\ast}\cup\{0\}.
\end{equation*}
Hence,
\begin{equation*}
\overline{\left(\widehat{\mu}+\widehat{\nu}\right)(\mathbb{Z})}\cup\{0\}=\sigma(\mu+\nu)\cup\{0\}=T^{\ast}\cup\{0\}.
\end{equation*}
If $0\in\sigma(\mu+\nu)$, then from Corollary $\ref{iz}$, $0\in \left(\widehat{\mu}+\widehat{\nu}\right)(\mathbb{Z})$ and we are done. Otherwise,
\begin{equation*}
\sigma(\mu+\nu)=T^{\ast}\subset\overline{\left(\widehat{\mu}+\widehat{\nu}\right)(\mathbb{Z})}
\end{equation*}
which finishes the proof.
\end{proof}
Now we can prove the promised theorem.
\begin{tw}
Let $\alpha\notin\left(\mathbb{Q}/2\pi\mathbb{Z}\right)$. Then $\delta_{\alpha}\notin\mathscr{S}$.
\end{tw}
\begin{proof}
Let $\mu$ be a Parreu measure. Without losing the generality we can assume that
$\sigma(\mu)\subset [0,2\pi)$ (we multiple original one by a constant and add a constant multiple of $\delta_{0}$ if necessary). Then by Lemma $\ref{row}$ we have $\exp(i\mu)\notin\mathscr{N}$. Moreover, by the spectral mapping theorem we get
$\sigma(\exp(i\mu))\subset T^{\ast}$. Let us consider now
$\rho_{1},\rho_{2}\in M(\mathbb{T})$ defined by the formulas
\begin{gather*}
\rho_{1}=\exp(i\mu)\ast\frac{\delta_{0}+\delta_{\pi}}{2}+\frac{\delta_{0}-\delta_{\pi}}{2}\ast\delta_{\alpha}\\
\rho_{2}=\exp(i\mu)\ast\frac{\delta_{0}-\delta_{\pi}}{2}+\frac{\delta_{0}+\delta_{\pi}}{2}\ast\delta_{\alpha}.
\end{gather*}
Since $(\delta_{0}+\delta_{\pi})\ast (\delta_{0}-\delta_{\pi})=0$
and for
$\nu_{1}=\frac{\delta_{0}-\delta_{\pi}}{2}\ast\delta_{\alpha}$,
$\nu_{2}=\frac{\delta_{0}+\delta_{\pi}}{2}\ast\delta_{\alpha}$ we
have
\begin{equation*}
\overline{\widehat{\nu_{1}}(\mathbb{Z})}=\overline{\widehat{\nu_{2}}(\mathbb{Z})}=T^{\ast}\cup\{0\}
\end{equation*}
we obtain $\rho_{1},\rho_{2}\in\mathscr{N}$ by Theorem
$\ref{dziu}$. Assume, towards the contradiction that
$\delta_{\alpha}\in\mathscr{S}$. Then, by definition
$\rho_{1}-\delta_{\alpha},\rho_{2}-\delta_{\alpha}\in\mathscr{N}$.
But
\begin{gather*}
\rho_{1}-\delta_{\alpha}=\exp(i\mu)\ast\frac{\delta_{0}+\delta_{\pi}}{2}-\frac{\delta_{0}+\delta_{\pi}}{2}\ast\delta_{\alpha}=\frac{\delta_{0}+\delta_{\pi}}{2}\ast(\exp(i\mu)-\delta_{\alpha}),\\
\rho_{2}-\delta_{\alpha}=\exp(i\mu)\ast\frac{\delta_{0}-\delta_{\pi}}{2}-\frac{\delta_{0}-\delta_{\pi}}{2}\ast\delta_{\alpha}=\frac{\delta_{0}-\delta_{\pi}}{2}\ast(\exp(i\mu)-\delta_{\alpha}).
\end{gather*}
Hence
$(\rho_{1}-\delta_{\alpha})\ast(\rho_{2}-\delta_{\alpha})=0$. It
follows from Lemma $\ref{zn}$ that
$(\rho_{1}-\delta_{\alpha})+(\rho_{2}-\delta_{\alpha})\in\mathscr{N}$.
However,
\begin{gather*}
(\rho_{1}-\delta_{\alpha})+(\rho_{2}-\delta_{\alpha})=\\
=\frac{\delta_{0}+\delta_{\pi}}{2}\ast(\exp(i\mu)-\delta_{\alpha})+\frac{\delta_{0}-\delta_{\pi}}{2}\ast(\exp(i\mu)-\delta_{\alpha})=\exp(i\mu)-\delta_{\alpha}.
\end{gather*}
And now, since we have assumed that
$\delta_{\alpha}\in\mathscr{S}$ we obtain
$\exp(i\mu)\in\mathscr{N}$ which is a contradiction.
\end{proof}
Let us recall that for $\mu\in M(\mathbb{T})$ and $\tau\in\mathbb{T}$ the convolution $\mu\ast\delta_{\tau}$ is a shift of measure given by the formula
\begin{equation*}
(\mu\ast\delta_{\tau})(E)=\mu(E-\tau)\text{ for all Borel sets }E\subset\mathbb{T}.
\end{equation*}
It follows from our results that even such a simple operation does not preserve the naturality of the spectrum.
\begin{cor}
For every $\tau\in\mathbb{T}\setminus\{0\}$ there exists $\mu\in\mathscr{N}$ such that $\mu\ast\delta_{\tau}\notin\mathscr{N}$.
\end{cor}
We omit the proof of the above corollary since it is a repetition of the arguments given in the proof of Fact $\ref{mm}$.
\\
Passing to sums of Dirac deltas will be done with the aid of some simple automorphisms of
the algebra $M(\mathbb{T})$. Let $m\in\mathbb{Z}$ be fixed and
consider $T_{m}:M(\mathbb{T})\mapsto M(\mathbb{T})$ given by the
following formula on the level of the Fourier-Stieltjes
transform:
\begin{equation*}
\widehat{T_{m}(\mu)}(n)=\widehat{\mu}(n-m)\text{ for all $\mu\in
M(\mathbb{T})$ and $n\in\mathbb{Z}$}.
\end{equation*}
It is clear that $T_{m}$ is an automorphism of the algebra
$M(\mathbb{T})$ and in addition can also be defined as follows (here we use the
Riesz representation theorem to obtain identification
$C(\mathbb{T})^{\ast}\simeq M(\mathbb{T})$)
\begin{equation*}
\int_{\mathbb{T}^{n}}f(t)d(T_{m}(\mu)(t))=\int_{\mathbb{T}}f(t)e^{imt}d\mu(t)\text{
for all $\mu\in M(\mathbb{T})$ and $f\in C(\mathbb{T})$}.
\end{equation*}
The above equation shows that automorphism $T_{m}$ sends a measure
$\mu$ to a measure absolutely continuous with respect to $\mu$
with density $e^{imt}$. This observation is a fundamental one when
we consider discrete measures since for every $\tau\in\mathbb{T}$
we have
\begin{equation*}
\int_{\mathbb{T}^{n}}f(t)d(T_{m}\delta_{\tau}(t))=\int_{\mathbb{T}}f(t)e^{imt}d\delta_{\tau}=f(\tau)e^{im\tau}\text{
for all $f\in C(\mathbb{T})$}
\end{equation*}
which leads to the formula
\begin{equation}
T_{m}\delta_{\tau}=e^{im\tau}\delta_{\tau}.
\end{equation}
Every automorphism of a Banach algebra preserves spectrum but
automorphisms $T_{m}$ also satisfy
$\widehat{(T_{m}\mu)}(\mathbb{Z})=\widehat{\mu}(\mathbb{Z})$ which
immediately proves the next fact.
\begin{fakt}
$\mu\in \mathscr{S}$ if and only if $T_{m}\mu\in\mathscr{S}$ for
every $m\in\mathbb{Z}$.
\end{fakt}
We are ready now to prove that finite sums of Dirac deltas are not
spectrally reasonable.
\begin{prop}\label{ds}
All finite sums of Dirac deltas are not spectrally reasonable
except constant multiples of $\delta_{0}$.
\end{prop}
\begin{proof}
Towards the contradiction let us take
\begin{equation*}
\mu=\sum_{k=1}^{N}a_{k}\delta_{\tau_{k}}\in\mathscr{S}\text{ where
$N\in\mathbb{N}$, $a_{k}\in\mathbb{C}$ and $\tau_{k}\in\mathbb{T}$
for $k\in\{1,\ldots,N\}$}.
\end{equation*}
We can assume that $\tau_{k}\neq \tau_{l}$ for $k\neq l$
and $a_{k}\neq 0$ for $k\in\{1,\ldots,N\}$. Without losing
generality $\tau_{k}\neq 0$ for $k\in\{1,\ldots,N\}$. Also we pick
the minimal $N$ for which such measure exists (the corresponding
measure is also denoted by $\mu$). Keeping in mind that $\tau_{1}\neq\tau_{2}$ we can
choose $m\in\mathbb{Z}$ for which $e^{im\tau_{1}}\neq
e^{im\tau_{2}}$. Then, for $\nu:=T_{m}(\mu)\in\mathscr{S}$ we have
\begin{equation*}
\nu=\sum_{k=1}^{N}a_{k}e^{im\tau_{k}}\delta_{\tau_{k}}\neq \mu.
\end{equation*}
Now,
\begin{equation*}
\mathscr{S}\ni
e^{im\tau_{1}}\mu-\nu=\sum_{k=2}^{N}a_{k}(e^{im\tau_{1}}-e^{im\tau_{k}})\delta_{\tau_{k}}=\sum_{k=1}^{N-1}a_{k+1}(e^{im\tau_{1}}-e^{im\tau_{k+1}})\delta_{\tau_{k+1}}\neq
0
\end{equation*}
which contradicts the choice of $N$.
\end{proof}
Dealing with infinite sums of Dirac deltas requires the next lemma which is also of independent interest. To formulate this result we introduce the following notation: for an algebraic polynomial $f$ with complex coefficients of the form
\begin{equation*}
f(z)=a_{n}z^{n}+a_{n-1}z^{n-1}+\ldots+a_{1}z+a_{0}
\end{equation*}
we will write $|f|_{1}$ for the sum of modulus of all coefficients, i.e.
\begin{equation*}
|f|_{1}:=|a_{0}|+|a_{1}|+\ldots+|a_{n-1}|+|a_{n}|.
\end{equation*}
\begin{lem}
Let $\alpha,\beta$ be two different complex numbers of modulus $1$. Then there exists an algebraic polynomial $f_{\alpha,\beta}$ with complex coefficients such that $f_{\alpha,\beta}(\alpha)=0$ and $f_{\alpha,\beta}(\beta)=|f|_{1}$.
\end{lem}
\begin{proof}
We will show first that there exist a polynomial $g$ with positive coefficients satisfying $g(\alpha\overline{\beta})=0$.
\\
Since $\alpha\neq\beta$ we have $\alpha\overline{\beta}\neq 1$. It is an elementary observation that there exists $N\in\mathbb{N}$ (which we choose minimal) with the property:
\begin{equation*}
0\in\mathrm{conv}\{(\alpha\overline{\beta})^{n}:n=0,1\ldots,N\}.
\end{equation*}
From the previous statement we deduce the existence of positive real numbers $a_{0},\ldots,a_{N}$ less or equal than $1$ satisfying
\begin{equation*}
0=a_{0}+a_{1}(\alpha\overline{\beta})+\ldots+a_{N}(\alpha\overline{\beta})^{N}.
\end{equation*}
Now, we define the polynomial $g$ as follows
\begin{equation*}
g(z):=a_{N}z^{N}+\ldots+a_{1}z+a_{0}, \text{ which gives $g(\alpha\overline{\beta})=0$ by the previous equation}.
\end{equation*}
Finally, let
\begin{equation*}
f_{\alpha,\beta}(z)=a_{0}+a_{1}\overline{\beta}z+\ldots+a_{N}(\overline{\beta})^{N}z^{N}.
\end{equation*}
Then $f_{\alpha,\beta}(\alpha)=g(\alpha\overline{\beta})=0$, also
\begin{equation*}
f_{\alpha,\beta}(\beta)=a_{0}+a_{1}+\ldots+a_{N}=|a_{0}|+|a_{1}|+\ldots+|a_{N}|
\end{equation*}
which finishes the proof.
\end{proof}
We are ready now to show that discrete measures are not reasonable.
\begin{tw}
Let $\mu\in M_{d}(\mathbb{T})$. Then $\mu$ is spectrally reasonable if and only if $\mu=c\delta_{0}$ for some $c\in\mathbb{C}$.
\end{tw}
\begin{proof}
Suppose on the contrary that there exists some $\mu\in M_{d}(\mathbb{T})\cap\mathscr{S}$ which is not of the form $c\delta_{0}$. From Proposition $\ref{ds}$ we obtain that $\mu$ is an infinite sum of Dirac deltas. Clearly, without losing the generality we can assume that $\mu$ has the following representation
\begin{equation*}
\mu=\delta_{\tau_{1}}+\sum_{n=2}^{\infty}a_{n}\delta_{t_{n}},
\end{equation*}
where $\tau_{n}\neq\tau_{l}$ for $n\neq l$, $a_{n},\tau_{n}\neq 0$ for all $n\in\mathbb{N}$, $|a_{2}|\geq |a_{3}|\geq\ldots$. Let $f_{e^{i\tau_{2}},e^{i\tau_{1}}}=:f_{1}$ be the polynomial from the previous lemma, i.e. satisfying $f_{1}(e^{i\tau_{2}})=0$ and $f_{1}(e^{i\tau_{1}})=|f_{1}|_{1}$. Then, we calculate easily for $f_{1}(z)=b_{0}+b_{1}z+\ldots+b_{k}z^{k}$ and any $\tau\in\mathbb{T}$:
\begin{equation*}
b_{0}\delta_{\tau}+b_{1}T_{-1}(\delta_{\tau})+\ldots+b_{k}T_{-k}(\delta_{\tau})=b_{0}\delta_{\tau}+b_{1}e^{i\tau}+\ldots+b_{k}e^{ik\tau}\delta_{\tau}=f_{1}(e^{i\tau})\delta_{\tau}.
\end{equation*}
Now, we have
\begin{gather*}
\nu:=b_{0}\mu+b_{1}T_{-1}(\mu)+\ldots+b_{k}T_{-k}(\mu)=\\
=f_{1}(e^{i\tau_{1}})\delta_{\tau_{1}}+a_{2}f_{1}(e^{i\tau_{2}})\delta_{\tau_{2}}+\sum_{n=3}^{\infty}a_{n}f_{1}(e^{i\tau_{n}})\delta_{\tau_{n}}=\\
=|f_{1}|_{1}\delta_{\tau_{1}}+\sum_{n=3}^{\infty}a_{n}f_{1}(e^{i\tau_{n}})\delta_{\tau_{n}}.
\end{gather*}
Moreover, $\nu\in\mathscr{S}$ which implies
\begin{equation*}
\frac{\nu}{|f_{1}|_{1}}=\delta_{\tau_{1}}+\sum_{n=3}^{\infty}a_{n}\frac{f_{1}(e^{i\tau_{n}})}{|f_{1}|_{1}}\delta_{\tau_{n}}\in\mathscr{S}.
\end{equation*}
Obviously, we can repeat this procedure to obtain
\begin{equation*}
\delta_{\tau_{1}}+\sum_{n=N}^{\infty}a_{n}\frac{f_{1}(e^{i\tau_{n}})}{|f_{1}|_{1}}\cdot\ldots\cdot\frac{f_{N-2}(e^{i\tau_{n}})}{|f_{N-2}|_{1}}\delta_{\tau_{n}}\in\mathscr{S}
\text{ for very $N\in\mathbb{N}$}.
\end{equation*}
However, since $\delta_{\tau_{1}}\notin\mathscr{S}$ and the set $\mathscr{S}$ is closed in $M(\mathbb{T})$ there exists $\varepsilon>0$ such that for every measure $\rho$ with $\|\rho\|<\varepsilon$ we have $\delta_{\tau_{1}}+\rho\notin\mathscr{S}$. This leads to a contradiction because we can choose $N\in\mathbb{T}$ large enough to obtain (to prove the first inequality we use the observation $|f_{j}(e^{i\tau_{n}})|\leq |f_{j}|_{1}$ for $j,n\in\mathbb{N}$)
\begin{equation*}
\|\sum_{n=N}^{\infty}a_{n}\frac{f_{1}(e^{i\tau_{n}})}{|f_{1}|_{1}}\cdot\ldots\cdot\frac{f_{N-2}(e^{i\tau_{n}})}{|f_{N-2}|_{1}}\delta_{\tau_{n}}\|\leq \sum_{n=N}^{\infty}|a_{n}|<\varepsilon.
\end{equation*}
\end{proof}
\section{Final Remarks}
\begin{enumerate}
  \item The results of sections 1 and 2 hold if $\mathbb{T}$ is replaced by any locally compact abelian group (proofs remains unchanged, essentially).
  \item The inclusion $\mathscr{C}\subset \mathscr{S}$ relies only on Zafran's theorem which is true for any compact abelian group but cannot be extended to non-compact case.
  \item The authors do not know if there exists any reasonable measure which is not highly reasonable. It seems quite reasonable to us that $\mathscr{S}$ would be just a unitization of the Zafran's ideal $\mathscr{C}$.
  \item It is very probable that for locally compact but non-compact $G$ the algebra $\mathscr{S}$ is trivial because of the result of O. Hatori ($\cite{h}$) which asserts that $\mathscr{N}(G)+L^{1}(G)=M(G)$ in this case which implies that even absolutely continuous measures need not be reasonable.
  \item The extension of results from Section 4 to $n$-dimensional torus seems to be only a technical challenge. However, for other compact groups (in particular, the torsion ones) this extension does not look straightforward.
\end{enumerate}

\end{document}